\documentclass[12pt,reqno]{amsart}
\usepackage[usenames]{color}
\usepackage[utf8x]{inputenc}
\usepackage[T1]{fontenc}
\usepackage{ae,aecompl}

\usepackage{pst-all}
\usepackage{pgf,pgfarrows,pgfnodes,pgfautomata,pgfheaps}
\usepackage{pstricks}
\usepackage{amssymb}
\usepackage{amsmath}
\usepackage{epsfig}

\usepackage{tikz}
\usepackage{enumerate}
\usepackage{verbatim}
\usepackage{hyperref}

\newcommand{\horedge}{\psline{*-*}(0,0)(1,0)}
\newcommand{\vertedge}{\psline{*-*}(0,0)(0,1)}
\newcommand{\east}{\psline(0,0)(1,0)}

\newcommand{\west}{\psline{o-*}(0,0)(-1,0)}
\newcommand{\neast}{\psline{o-*}(0,0)(1;60)}
\newcommand{\nwest}{\psline{*-o}(0,0)(1;120)}

\usepackage{graphicx}

\newtheorem{example}{Example}[section]

\newtheorem{remark}[example]{Remark}
\newtheorem{theorem}[example]{Theorem}

\newtheorem{conjecture}[example]{Conjecture}

\newtheorem{proposition}[example]{Proposition}

\def\Proof{\noindent \it Proof -- \rm}
\def\qed{\hspace{3.5mm} \hfill \vbox{\hrule height 3pt depth 2 pt width 2mm}
\bigskip}

\newcommand{\pref}[1]{(\ref{#1})}

\newcommand{\Z}{\mathbb Z}
\newcommand{\R}{{\bf\mathcal R}}

\author{Jean-Christophe Aval, Philippe Duchon}
\title[HTFPL$s$: rare couplings and tilings of hexagons]{Half-turn symmetric FPL$s$ with rare couplings and tilings of hexagons}
\address{LaBRI, Universit\'e de Bordeaux,
351 cours de la Lib\'eration, 33405 Talence cedex, France}
\date{\today}

\begin{document}
\maketitle
\SpecialCoor

\begin{abstract}
In this work, we put to light a formula that relies the number of fully packed loop configurations (FPLs) 
associated to a given coupling $\pi$ to the number of half-turn symmetric FPLs (HTFPLs) of even size whose coupling is a  punctured version
of the coupling $\pi$. When the coupling $\pi$ is the coupling with all arches parallel $\pi_0$ (the ``rarest'' one), this formula states the equality of 
the number of  corresponding HTFPLs to the number of cyclically-symmetric plane partition of the same size. 
We provide a bijective proof of this fact. 
In the case of HTFPLs odd size, and although there is no similar expression, we study the number of HTFPLs
whose coupling is a slit version of $\pi_0$, and put to light
new  puzzling enumerative coincidence involving countings of tilings of hexagons and various symmetry classes of FPLs.
\end{abstract} 

%%%%%%%%%%%%%%%%%%%%%%%%%%%%%%%%%%%%%%%%%%%%%%%%%%%%%%%%%%%%%%%%%%%%%%%%%%%%%%%
\section*{Introduction}

Fully packed loop configurations (FPLs) are ubiquitous objects which are fascinating both in the world of theoretical physics
(they appear in the so-called six-vertex ice model) and in the world of combinatorics (they are in bijection with alternating sign matrices,
which are the center of an intense research for years).
In 2004 Razumov and Stroganov \cite{RS} stated a remarkable conjecture that relies the stationary distribution of the 
$O(1)$-dense loop model to the enumeration of FPLs according to their coupling. After several years of efforts, their formula was
only recently proved by Cantini and Sportiello \cite{CS} by means of a purely combinatorial method using the operation of gyration
discovered by Wieland \cite{wieland}.
Following Razumov and Stroganov's investigations, de Gier \cite{dG} gave in 2005 an analogous conjectural formula for the same 
model with half-turn symmetry constraints.

When we compare Razuomv-Stroganov's and de Gier's formula (for the even size), we are led to the following interesting expression:
the number of  FPLs of size $n$ and  coupling $\pi$ is equal to the sum of the numbers of 
half-turn symmetric FPLs (HTFPLs) of size $2n$ and coupling a  punctured version of $\pi$. 
A special case is when the coupling is the rarest one $\pi_0$ (with the arches all 
parallel), where this expression reduces to an equality between the number of half-turn symmetric FPLs of size $2n$
with their coupling being a punctured version of $\pi_0$ and the number of cyclically symmetric plane partition of size $2n$.
We are able to prove this assertion {\em bijectively}.

In the case of the odd size, there is no natural expression between couplings of HTFPLs and asymmetric FPLs.
Nevertheless, we may study the number of HTFPLs of size $2n+1$ whose couplings are slit versions of $\pi_0$.
Using a factorization principle due to Ciucu \cite{ciucu1}, we are lead to evaluate the number of tilings with losenges 
of portions of some hexagonal regions. These numbers of tilings may be expressed through determinants \cite{kratt}. 
Surprisingly, we put to light that several determinant expressions are proved or conjectured to be equal 
to the number  of symmtry classes of FPLs!

This paper is organized as follows: Section \ref{sec:def} presents all definitions relative to FPLs and their couplings,
Section \ref{sec:even} deals with the case of even-sized HTFPLs, Section \ref{sec:odd} presents the problem studied
in the case of the odd size, together with its reduction to the evaluation of determinants and presents new intriguing results and conjectures
of equinumeration between certain tilings and symmetry classes of FPLs.

%%%%%%%%%%%%%%%%%%%%%%%%%%%%%%%%%%%%%%%%%%%%%%%%%%%%%%%%%%%%%%%%%%%%%%%%%%%%%%%
\section{Definitions}\label{sec:def}

\subsection{FPLs and their couplings}

A {\em fully-packed loop configuration} (FPL for short) of size $N$ is a subgraph of the $N \times N$ square lattice,
where each internal vertex has degree exactly $2$. The set of edges forms a set of closed loops and paths ending at the
boundary vertices. The boundary conditions are the alternating conditions: boundary vertices also have
degree $2$ when boundary edges (edges that connect the finite square lattice to the rest of the $\Z^2$ lattice)
are taken into account, and these boundary edges, when going around the grid, are alternatingly “in” and
“out” of the FPL. For definiteness, we use the convention that the top edge along the left border is always
“in”. Thus, exactly $2N$ boundary edges act as endpoints for paths, and the FPL consists of $N$ noncrossing
paths and an indeterminate number of closed loops.

Any FPL $f$ of size $N$ has a {\em coupling} $\pi(f)$, which is a partition of the set of integers $\{1\dots 2N\}$ into pairs,
defined as follows: first label the endpoints of the open loops $1$ to $2N$ in clockwise or counterclockwise
order (for definiteness, we use counterclockwise order, starting with the top left endpoint); then the link
pattern $\pi(f)$ will include pair $(i, j)$ if and only if $f$ contains a loop whose two endpoints are labeled $i$
and $j$. Because the loops are noncrossing, the coupling satisfies the noncrossing condition: if a link
pattern $\pi$ contains two pairs $(i, j)$ and $k,l$, then one cannot have $i < k < j < l$. The possible link
patterns for FPLs of size $N$ are thus counted by the Catalan numbers $C_N = \frac1{N+1}{2N\choose N}$.
Figure~\ref{fig:FPL} gives an example of an FPL together with its coupling.
We shall denote by $A(N;\pi)$ the number of FPLs of size$N$ which afford coupling $\pi$, and by $A(N)$
the total number of FPLs, which is equal, because of the bijection between FPL and alternating-sign matrices 
\cite{zeilberger} to:
\begin{equation}
A(N)=\prod_{i=0}^{n−1} \frac{(3i + 1)!}{(n + i)!}\,.
\end{equation}

\begin{figure}[ht]

\begin{tabular}{cc}

      \psset{unit=4mm}
      \begin{pspicture}(-1,0)(9,9)
      \psline(1,0)(1,1)(2,1)(2,2)(3,2)(3,4)(2,4)(2,3)(1,3)(1,2)(0,2)
      \psline(0,4)(1,4)(1,6)(0,6)
      \psline(0,8)(1,8)(1,7)(3,7)(3,8)(2,8)(2,9)
      \psline(3,0)(3,1)(5,1)(5,0)
      \psline(7,0)(7,1)(6,1)(6,3)(7,3)(7,2)(8,2)(8,1)(9,1)
      \psline(9,3)(8,3)(8,4)(4,4)(4,5)(2,5)(2,6)(4,6)(4,9)
      \psline(9,5)(8,5)(8,6)(7,6)(7,5)(5,5)(5,6)(6,6)(6,7)(5,7)(5,8)(6,8)(6,9)
      \psline(9,7)(7,7)(7,8)(8,8)(8,9)
      \psline(4,2)(4,3)(5,3)(5,2)(4,2)
      \rput[r](0,8){$1$}\rput[r](0,6){$2$}
    \end{pspicture}

&

 \begin{pspicture}(0,0)
 \rput(.5,3){$1$}\rput(.1,2.4){$2$}
 \end{pspicture}
 \epsfig{file=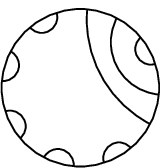,scale=2}

\end{tabular}

\caption{An FPL of size 8 and its coupling}
\label{fig:FPL}
\end{figure}

Let us introduce a particular coupling, denoted $\pi_{0,n}$ (or $\pi_{0}$ if there is no ambiguity),
defined as:
$$\pi_{0,n} = \{ \{i,2n+1-i\}_{1\leq i\leq n}\}.$$
The coupling $\pi_0$ is, up to rotation the rarest one: $A(n,\pi_0)=1$.

The $2N$ generators $e_1, \dots, e_{2N}$ of the cyclic Temperley-Lieb algebra act on couplings of size $N$
in the following way: if the coupling $\pi$ contains pairs $(i, j$) and $(i + 1, k)$, then $e_i\pi = \pi'$, where $\pi'$
is obtained from $\pi$ by replacing the pairs $(i, j)$ and $(i + 1, k)$ by $(i, i + 1)$ and $(j, k)$; if $(i, i + 1) \in \pi$,
then $\pi' = \pi$. An illustration of this action is given by Figure~\ref{fig:TLaction}.

\begin{figure}[ht]
\begin{center}
  \begin{pspicture}(9,4)
    \SpecialCoor
    \rput(2,2){%
      \pscircle[linewidth=.2pt](0,0){1}
      \multido{\n=00.0+22.5}{16}{\rput(1;\n){\psdot[dotsize=.1]}}
      \psarc(1.0196;11.25){.1989}{112.5}{270}
      \psarc(1.0196;56.25){.1989}{157.5}{315}
      \psarc(1.0196;123.75){.1989}{225}{22.5}
      \psarc(1.0196;191.25){.1989}{292.5}{90}
      \psarc(1.0196;281.25){.1989}{22.5}{180}
      \psarc(1.2027;191.25){.6682}{305}{67.5}
      \psarc(1.2027;281.25){.6682}{45}{157.5}
      \psarc(1.8;33.75){1.4966}{180}{247.5}
      \rput(0,0){$\pi$}
      \psset{linecolor=blue}
	\pscircle(0,0){2}
	\multido{\n=180.0+022.5}{14}{\psline(1;\n)(2;\n)}
	\psarc(1.0196;146.25){.1989}{45}{247.5}
	\psarc(2.0392;146.25){.3978}{247.5}{45}
	\rput(2.2;135){$i$}
	\rput[r](2.2;157.5){$i+1$}
	\rput(2.2;112.5){$j$}
	\rput(2.2;225){$k$}
	\multido{\nn=00.0+22.5}{16}{\rput(2;\nn){\psdot[dotsize=.1,linecolor=blue]}}
      }% rput(2,2)
      \rput(7,2){
	\pscircle(0,0){2}
	\multido{\n=00.0+22.5}{16}{\rput(2;\n){\psdot[dotsize=.1]}}
	\psarc(2.0392;11.25){.3978}{112.5}{270}
	\psarc(2.0392;56.25){.3978}{157.5}{315}
	\psarc(2.0392;146.25){.3978}{247.5}{45}
	\psarc(2.0392;191.25){.3978}{292.5}{90}
	\psarc(2.0392;281.25){.3978}{22.5}{180}
	\psarc(2.4054;281.25){1.3364}{45}{157.5}
	\psarc(3.6;33.75){2.9932}{180}{247.5}
	\psarc(3.6;168.75){2.9932}{315}{22.5}
	\rput(0,0){$e_i(\pi)$}
      }
   \end{pspicture}
\end{center}
     
\caption{Temperley-Lieb action}     
\label{fig:TLaction}
\end{figure}
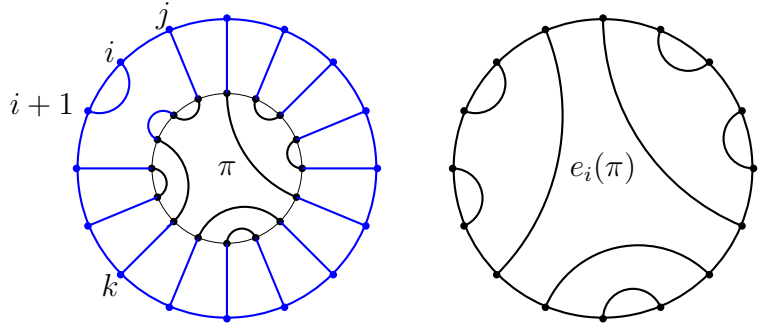

One can define a Markov chain on couplings
where we choose at each time step one of the appropriate generators (uniformly at random) and apply it
to the current state. The Markov chain defined in this way is easily checked to be irreducible and aperiodic, hence
it has a unique stationary distribution. The celebrated Razumov-Stroganov conjecture \cite{RS}, proven by Cantini
and Sportiello \cite{CS}, may be stated as follows.

\begin{theorem}\label{theo:RSCS}{\em [Cantini, Sportiello]} 
The stationary distribution for couplings of size $N$ is
\begin{equation}\label{eq:RS}
\mu(\pi) = \frac{A(N; \pi)} {A(N)}.
\end{equation}
\end{theorem}

\subsection{HTFPLs -- Punctured and slit couplings}

An FPL is said to be {\em half-turn symmetric} if it is invariant under the central symmetry
of the square grid. It is easy to observe that such HTFPLs do exist whatever the parity of the size $N$.
Let us denote by $A_{HT}(N)$ the number of HTFPLs of size $N$.

HTFPLs, be they of even or odd size, have couplings that are invariant
under a half-turn rotation: if their size is $L$ and the coupling
contains an edge $(i,j)$, it must also contain $(i+L,j+L)$.

For odd $L$, parity and planarity considerations immediately imply
that the coupling must contain exactly one diameter edge of the form
$(i,i+L)$, with the endpoints $i+1$ to $i+L-1$ organized into a normal
coupling (and endpoints $L+i+1$ to $i-1$ organized into a translated
version of the same). Such a coupling of size $2L$ can be represented
more compactly as a ``slit'' coupling of (odd) size $L$, where the
diameter edge becomes a singleton $(i)$ and each pair of edges $(j,k)$
and $(j+L,k+L)$ becomes a single $(j,k)$ edge. Graphically, this
corresponds to a classical coupling of size $L-1$ with an added single
vertex (which we represent by a half-edge leading inside the circle).

For even $L$, no diagonal edge can exist for parity reasons. Instead,
HT-symmetric couplings of size $2L$ can be represented as classical
plane couplings of size $L$ drawn on a punctured disk (or
half-cylinder) instead of a full disk.

Figure \ref{fig:HTcoupling} shows examples of  half-turn symmetric couplings
respectively of odd (left) and even (right) size.

\begin{figure}[ht]
\begin{tabular}{cc}
	\epsfig{file=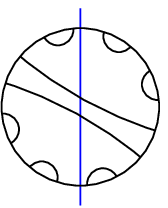, scale=2}
&
	\epsfig{file=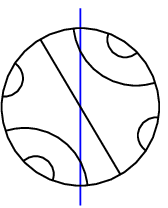, scale=2}
\end{tabular}
\caption{Half-turn symmetric couplings  of odd (left) and even (right) size}
\label{fig:HTcoupling}
\end{figure}

Let us denote by $A_{HT}(N; \pi)$ the number of HTFPLs which have $\pi$ 
as coupling.

Similarly to the asymmetric case, 
and for $N \ge 2$, we consider the $N$ “symmetrized” operators 
\begin{equation}
e'_i = ei e_{i+N}.  
\end{equation}

These operators act on the couplings of HTFPLs of size $N$, we may define a Markov chain on the set of half-turn
symmetric couplings. The assertion analogous to Theorem \ref{theo:RSCS} is due to de Gier \cite{dG} and may be stated as follows.

\begin{conjecture}\label{conj:dG}{\em[de Gier]}
The stationary distribution for couplings of size $N$ is
\begin{equation}\label{eq:dG}
\mu_{HT}(\pi) = \frac{A_{HT}(N; \pi)} {A_{HT}(N)}.
\end{equation}
\end{conjecture}

%%%%%%%%%%%%%%%%%%%%%%%%%%%%%%%%%%%%%%%%%%%%%%%%%%%%%%%%%%%%%%%%%%%%%%%%%%%%%%%
\section{Even-sized HTFPLs with rare couplings}\label{sec:even}

\subsection{A general formula}

When viewed as ``punctured'' plane couplings, the couplings of
even-sized HTFPLs have a natural projection to ``normal'' plane
couplings of half their size - the projection corresponds to simply
forgetting the puncture. What is more important, this projection
commutes with the $e_i$ and $e'_i$ operators: if $\pi'$ is a punctured
plane coupling and $p$ is the projection from punctured to normal
plane couplings, one has $p(e'_i(\pi')) = e_i(p(\pi'))$.

An immediate consequence is that the eigenvector for the $H'$
Hamiltonian must project to the eigenvector for $H$. In terms of FPL
and HTFPL enumerations, in light of (\ref{eq:RS}) and assuming
(\ref{eq:dG}), this becomes, for any coupling $\pi$:
\begin{equation}\label{eq:refined}
  \frac{A(n;\pi)}{A(n)} = \sum_{\pi'} \frac{A_{HT}(2n;\pi')}{A_{HT}(2n)},
\end{equation}
where the sum in the right-hand side extends to all punctured
couplings $\pi'$ such that $p(\pi')=\pi$.

Now, it is known that $A_{HT}(2n) = P_{SC}(2n) A(n)$, where
$P_{SC}(2n)$ denotes the number of cyclically symmetric plane
partitions of size $2n$. Thus, (\ref{eq:refined}) is equivalent to
\begin{equation}\label{eq:ref}
  \sum_{\pi'} A_{HT}(2n;\pi') = P_{SC}(2n) A(n;\pi)
\end{equation}
with the same convention on the summation.

\subsection{The case of the rarest coupling}

When $\pi$ is one of the rotated versions of the rarest coupling, one
has $A(n;\pi)=1$ and (\ref{eq:ref}) simplifies accordingly. Our first
result is a bijective proof of this special case of equation (\ref{eq:ref}).

\begin{theorem}
For any integer $n$, there exists a bijection between the set of
HTFPLs of size $2n$ whose coupling is a punctured version $\pi_{0,n}$,
and cyclically symmetric plane partitions of size $2n$.
\end{theorem}

\begin{proof}
  The first thing to do is identify exactly which punctured couplings
  project to $\pi_{0,n}$. As plane couplings of size $4n$, these must
  link $1$ to either $2n$ or $4n$, and $2n+1$ with the other, and more
  generally, for each $1\leq k\leq n$, $k$ must be linked with either
  $2n+1-k$ or $4n+1-k$, and $2n+k$ must be linked with the other. If
  we add the noncrossing condition, we obtain a full description of
  the $n+1$ possible couplings:
  \begin{displaymath}
    \pi'_{k,n} = \{ \{i,4n+1-i\}_{1\leq i\leq k}\} \cup \{\{i,2n+1-i\}_{k<i\leq n}\}
  \end{displaymath}
  where $k$ ranges from $0$ to $n$.

Now, the important property of this set of plane couplings is that
they are exactly \emph{all} plane couplings of size $4n$ whose short
links are among $\{1,4n\},\{n,n+1\},\{2n,2n+1\},\{3n,3n+1\}$ - in
fact, except for $\pi'_{0,n}$ and $\pi'_{n,n}$, these are exactly all
short links of each $\pi'_{k,n}$. On the corresponding (HT)FPLs, this
translates into exactly the same set of {\em fixed edges} 
(we refer to \cite{CKLN} for the presentation of the fixed edges technique): 

\begin{itemize}
\item all eastbound edges from odd vertices in the $(i>j,i+j<2n-1)$
  area;
\item all edges obtained by rotations from the previous: northbound
  from even vertices in the $(i>j,i+j>2n-1)$ area, westbound from odd
  vertices with $(i<j,i+j>2n-1)$, and southbound from even vertices
  with $(i<j,i+j<2n-1)$.
\end{itemize}

The fixed edges for size $12$ are shown in Figure~\ref{fig:fixed_even} (left).

\begin{figure}[ht]
  \begin{tabular}{cc}
\psset{unit=4mm}
    \begin{pspicture}(0,0)(13,13)
\multido{\i=2+2}{6}{\rput(0,\i){\horedge}}
\multido{\i=3+2}{5}{\rput(1,\i){\horedge}}
\multido{\i=4+2}{4}{\rput(2,\i){\horedge}}
\multido{\i=5+2}{3}{\rput(3,\i){\horedge}}
\multido{\i=6+2}{2}{\rput(4,\i){\horedge}}
\rput(5,7){\horedge}
\multido{\i=1+2}{6}{\rput(12,\i){\horedge}}
\multido{\i=2+2}{5}{\rput(11,\i){\horedge}}
\multido{\i=3+2}{4}{\rput(10,\i){\horedge}}
\multido{\i=4+2}{3}{\rput(9,\i){\horedge}}
\multido{\i=5+2}{2}{\rput(8,\i){\horedge}}
\rput(7,6){\horedge}
\multido{\i=1+2}{6}{\rput(\i,0){\vertedge}}
\multido{\i=2+2}{5}{\rput(\i,1){\vertedge}}
\multido{\i=3+2}{4}{\rput(\i,2){\vertedge}}
\multido{\i=4+2}{3}{\rput(\i,3){\vertedge}}
\multido{\i=5+2}{2}{\rput(\i,4){\vertedge}}
\rput(6,5){\vertedge}
\multido{\i=2+2}{6}{\rput(\i,12){\vertedge}}
\multido{\i=3+2}{5}{\rput(\i,11){\vertedge}}
\multido{\i=4+2}{4}{\rput(\i,10){\vertedge}}
\multido{\i=5+2}{3}{\rput(\i,9){\vertedge}}
\multido{\i=6+2}{2}{\rput(\i,8){\vertedge}}
\rput(7,7){\vertedge}
\psline[linestyle=dotted](.5,.5)(12.5,12.5)
\psline[linestyle=dotted](.5,12.5)(12.5,.5)
\end{pspicture}
&\ \ 
\psset{unit=4mm}
\begin{pspicture}(11,13)
\rput(0,1){
\multido{\i=0+1}{11}{\rput(0,\i){\vertedge}}
\multido{\i=1+1}{10}{\rput(1,\i){\vertedge}}
\multido{\i=2+1}{8}{\rput(2,\i){\vertedge}}
\multido{\i=3+1}{6}{\rput(3,\i){\vertedge}}
\multido{\i=4+1}{4}{\rput(4,\i){\vertedge}}
\multido{\i=5+1}{2}{\rput(5,\i){\vertedge}}
\rput(2,0){\vertedge}
\rput(4,0){\vertedge}
\rput(3,1){\vertedge}
\rput(5,1){\vertedge}
\rput(4,2){\vertedge}
\rput(5,3){\vertedge}
\rput(5,8){\vertedge}
\rput(4,9){\vertedge}
\rput(3,10){\vertedge}
\rput(5,10){\vertedge}
\psline(0,0)(5.5,0)
\psline(0,1)(5.5,1)
\psline(1,2)(5.5,2)
\psline(2,3)(5.5,3)
\psline(3,4)(5.5,4)
\psline(4,5)(5.5,5)
\psline(5,6)(5.5,6)
\psline(4,7)(5.5,7)
\psline(3,8)(5.5,8)
\psline(2,9)(5.5,9)
\psline(1,10)(5.5,10)
\psline(0,11)(5.5,11)
\multido{\i=3+2}{4}{\rput(0,\i){\east}}
\multido{\i=4+2}{3}{\rput(1,\i){\east}}
\multido{\i=5+2}{2}{\rput(2,\i){\east}}
\rput(3,6){\east}
\rput(1,0){\horedge}\rput(3,0){\horedge}
\rput(4,0){\horedge}\rput(2,11){\horedge}
\rput(4,11){\horedge}
\psarc(5.5,5.5){.5}{-90}{90}
\psarc(5.5,5.5){1.5}{-90}{90}
\psarc(5.5,5.5){2.5}{-90}{90}
\psarc(5.5,5.5){3.5}{-90}{90}
\psarc(5.5,5.5){4.5}{-90}{90}
\psarc(5.5,5.5){5.5}{-90}{90}
}
\end{pspicture}

  \end{tabular}
  \caption{Fixed and non-fixed edges in even size}
  \label{fig:fixed_even}
\end{figure}
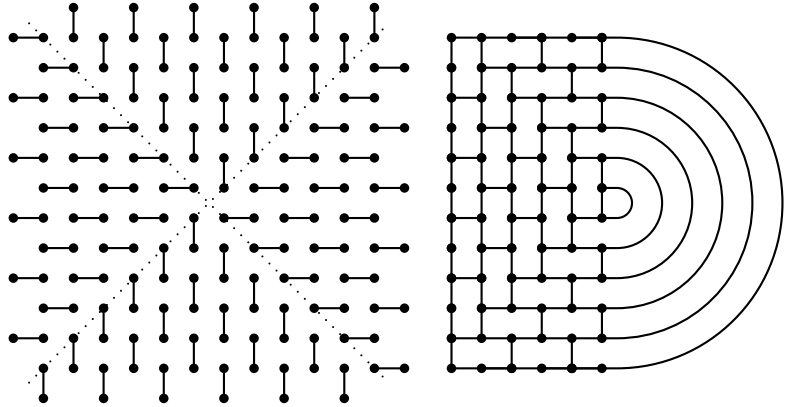

It is easy to check that all HTFPLs with these edges will have one of
the $\pi'_{k,n}$ as their coupling; more precisely, though it is not
important for our purpose, any FPL, whether half-turn-symmetric or
not, will have such a coupling.

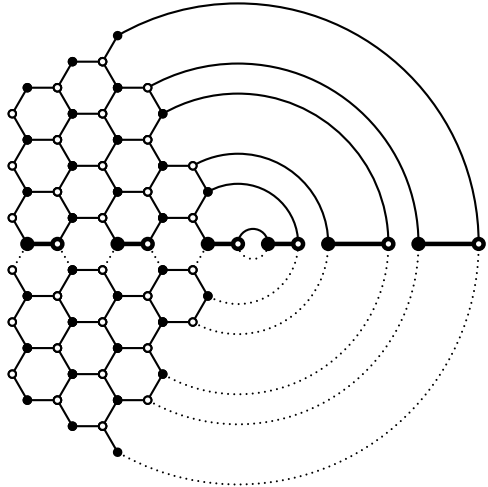
\begin{figure}[ht]
\psset{unit=4mm
,dotsep=1.5pt
}
\begin{pspicture}(16,16)
\rput(8,8){

\psarc[linestyle=dotted](.5,0){.5}{-180}{0}
\psarc(.5,0){.5}{0}{180}
\psarc[linestyle=dotted](0,0){2}{-120}{0}
\psarc(0,0){2}{0}{120}
\psarc[linestyle=dotted](0,0){3}{-120}{0}
\psarc(0,0){3}{0}{120}
\psarc[linestyle=dotted](0,0){5}{-120}{0}
\psarc(0,0){5}{0}{120}
\psarc[linestyle=dotted](0,0){6}{-120}{0}
\psarc(0,0){6}{0}{120}
\psarc[linestyle=dotted](0,0){8}{-120}{0}
\psarc(0,0){8}{0}{120}
\psline[linewidth=.15]{o-*}(2,0)(1,0)
\psline[linewidth=.15]{o-*}(5,0)(3,0)
\psline[linewidth=.15]{o-*}(8,0)(6,0)

\rput(1;300){\rput(-10,0){\rput(5;300){\nwest}}
\rput(-9,0){\rput(3;300){\nwest}\rput(6;300){\nwest}}
\rput(-8,0){\multido{\i=1+3}{3}{\rput(\i;300){\nwest}}}
\rput(-7,0){\multido{\i=-1+3}{3}{\rput(\i;300){\nwest}}}
%\rput(-6,0){\multido{\i=-3+3}{3}{\rput(\i;300){\nwest}}}
\rput(-6,0){\psline{*-o}(3;120)(4;120)\psline{*-o}(3;300)(2;300)
  \psline[linestyle=dotted]{*-o}(0,0)(1;120)}
\rput(-5,0){\multido{\i=-5+3}{4}{\rput(\i;300){\nwest}}}
\rput(-4,0){\multido{\i=-4+3}{3}{\rput(\i;300){\nwest}}}
\rput(-3,0){\multido{\i=-6+3}{2}{\rput(\i;300){\nwest}}}
\rput(-3,0){\psline[linestyle=dotted]{o-*}(1;120)(0,0)}
\rput(-2,0){\multido{\i=-5+3}{3}{\rput(\i;300){\nwest}}}
\rput(-1,0){\multido{\i=-7+3}{3}{\rput(\i;300){\nwest}}}
\multido{\i=-6+3}{2}{\rput(\i;300){\nwest}}
}

\rput(7;120){\rput(-1,0){\west}}
\rput(6;120){\rput(-3,0){\west}\west}
\rput(5;120){\rput(-2,0){\west}}
\rput(4;120){\rput(-1,0){\west}\rput(-4,0){\west}}
\rput(3;120){\rput(-3,0){\west}\west}
\rput(2;120){\rput(-2,0){\west}\rput(-5,0){\west}}
\rput(1;120){\rput(-1,0){\west}\rput(-4,0){\west}}
\multido{\i=-6+3}{3}{\rput(\i,0){\psline[linewidth=.15]{o-*}(0,0)(-1,0)}}
\rput(7;240){\rput(-1,0){\west}}
\rput(6;240){\rput(-3,0){\west}\west}
\rput(5;240){\rput(-2,0){\west}}
\rput(4;240){\rput(-1,0){\west}\rput(-4,0){\west}}
\rput(3;240){\rput(-3,0){\west}\west}
\rput(2;240){\rput(-2,0){\west}\rput(-5,0){\west}}
\rput(1;240){\rput(-1,0){\west}\rput(-4,0){\west}}

\rput(-10,0){\rput(5;60){\neast}}
\rput(-9,0){\rput(3;60){\neast}\rput(6;60){\neast}}
\rput(-8,0){\multido{\i=1+3}{3}{\rput(\i;60){\neast}}}
\rput(-7,0){\multido{\i=2+3}{2}{\rput(\i;60){\neast}}}
\rput(-7,0){\psline[linestyle=dotted]{*-o}(0,0)(1;240)}
\rput(-6,0){\multido{\i=-3+3}{3}{\rput(\i;60){\neast}}}
\rput(-5,0){\multido{\i=-5+3}{4}{\rput(\i;60){\neast}}}
%\rput(-4,0){\multido{\i=-4+3}{3}{\rput(\i;60){\neast}}}
\rput(-4,0){\psline{o-*}(4;240)(3;240)\psline{o-*}(2;60)(3;60)\psline[linestyle=dotted]{o-*}(1;240)(0,0)}
\rput(-3,0){\multido{\i=-6+3}{3}{\rput(\i;60){\neast}}}
\rput(-2,0){\multido{\i=-5+3}{3}{\rput(\i;60){\neast}}}
\rput(-1,0){\multido{\i=-7+3}{2}{\rput(\i;60){\neast}}}
\rput(-1,0){\psline[linestyle=dotted]{o-*}(1;240)(0,0)}
\multido{\i=-6+3}{2}{\rput(\i;60){\neast}}

}
\end{pspicture}

  \caption{Honeycomb lattice version of the graph in Figure~\ref{fig:fixed_even}}
  \label{fig:honey_12}
\end{figure}

Thus, our problem becomes that of finding a bijection between the set
of HTFPLs with this set of fixed edges, and CSPPs of size $2n$. This
is relatively straightforward: since each vertex in the grid is
incident to one fixed edge, these (HT)FPLs are in a natural bijection
with the (half-turn-symmetric) perfect matchings of the subgraph of
non-fixed edges. Taking symmetry into account corresponds to taking
the quotient of the graph under the half-turn-symmetry, and it is easy
to check that this quotient graph is also the quotient under a
$2\pi/3$-rotation of a hexagonal region (of size $2n$) of the
honeycomb lattice. In other words, the HTFPLs with these fixed edges
are in bijection with the perfect matchings of the honeycomb lattice
that are invariant under a third-turn rotation - or, taking the dual,
lozenge tilings of a regular hexagon (of side $2n$) that are invariant
under a rotation of order $3$, that is, cyclically symmetric plane
partitions of size $2n$.
\end{proof}

%%%%%%%%%%%%%%%%%%%%%%%%%%%%%%%%%%%%%%%%%%%%%%%%%%%%%%%%%%%%%%%%%%%%%%%%%%%%%%%
\section{Rare couplings in odd size}\label{sec:odd}

\subsection{Factorization}

While there is an easy way to project couplings of HTFPLs of odd size
$2n+1$ to those of FPLs of size $n$ (by ``unslitting'' them), this
projection, contrary to the even sized case, does not commute with the
$e_i$ and $e'_i$ operators, so that (\ref{eq:RS}) and (\ref{eq:dG})
together do not have a ``nice'' consequences on the numbers
$A_{HT}(2n+1;\pi')$ and $A(n;\pi)$. Still, applying the fixed edges
technique to some sets of HTFPLs whose couplings are a slit version of
the rarest coupling does lead to intriguing enumerative results.

The slit couplings we are looking for are all those with at most two
short edges, \textit{i.e.} of the form (for some $0\leq k\leq n+1$)
\begin{displaymath}
  \left\{\{i,2n+2-i\}_{1\leq i\leq k}\right\}\cup \left\{\{k\}\right\} \cup 
  \left\{\{i,2n+3-i\}_{k+1\leq i\leq n+1}\right\}
\end{displaymath}
or rotated from this form. In extended form, these are all
HT-symmetric couplings with at most $4$ short edges that are
restricted to be in positions $(i,i+1)$, $(n+i+1,n+i+2)$,
$(2n+i+1,2n+i+2)$ and $(3n+i+2,3n+i+3)$ for some $i$.

If we use the case $i=0$ above (that is, we allow short edges
$(4n+2,1)$, $(n+1,n+1)$, $(2n+1,2n+2)$ and $(3n+2,3n+3)$), we get a
large set of fixed edges that is similar to what we got in the
even-sized case:
\begin{itemize}
\item eastbound edges from odd vertices in the $(i>j,i+j<2n)$ area,
  and westbound edges from odd vertices in the symmetric
  $(i<j,i+j>2n)$ area;
\item northbound edges from even vertices in the $(i>j,i+j>2n)$ area,
  and southbound edges from even vertices in the symmetric
  $(i<j,i+j<2n)$.
\end{itemize}

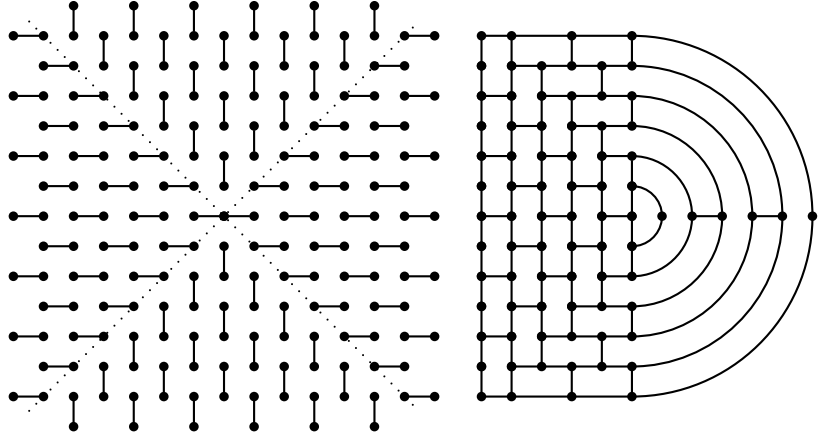
\begin{figure}[ht]
  \begin{tabular}{cc}
\psset{unit=4mm}
    \begin{pspicture}(14,14)
\multido{\i=1+2}{7}{\rput(0,\i){\horedge}\rput(13,\i){\horedge}}
\multido{\i=2+2}{6}{\rput(1,\i){\horedge}\rput(12,\i){\horedge}}
\multido{\i=3+2}{5}{\rput(2,\i){\horedge}\rput(11,\i){\horedge}}
\multido{\i=4+2}{4}{\rput(3,\i){\horedge}\rput(10,\i){\horedge}}
\multido{\i=5+2}{3}{\rput(4,\i){\horedge}\rput(9,\i){\horedge}}
\multido{\i=6+2}{2}{\rput(5,\i){\horedge}\rput(8,\i){\horedge}}
\rput(6,7){\horedge}\rput(7,7){\horedge}
\multido{\i=2+2}{6}{\rput(\i,0){\vertedge}\rput(\i,13){\vertedge}}
\multido{\i=3+2}{5}{\rput(\i,1){\vertedge}\rput(\i,12){\vertedge}}
\multido{\i=4+2}{4}{\rput(\i,2){\vertedge}\rput(\i,11){\vertedge}}
\multido{\i=5+2}{3}{\rput(\i,3){\vertedge}\rput(\i,10){\vertedge}}
\multido{\i=6+2}{2}{\rput(\i,4){\vertedge}\rput(\i,9){\vertedge}}
\rput(7,5){\vertedge}\rput(7,8){\vertedge}
\psline[linestyle=dotted](.5,.5)(13.5,13.5)
\psline[linestyle=dotted](.5,13.5)(13.5,.5)
\end{pspicture}
&\ \ 
\psset{unit=4mm}
\begin{pspicture}(12,14)
\rput(0,1){
\multido{\i=0+1}{12}{\rput(0,\i){\vertedge}\rput(1,\i){\vertedge}}
\multido{\i=1+1}{10}{\rput(2,\i){\vertedge}}
\multido{\i=2+1}{8}{\rput(3,\i){\vertedge}}
\multido{\i=3+1}{6}{\rput(4,\i){\vertedge}}
\multido{\i=4+1}{4}{\rput(5,\i){\vertedge}}
\rput(3,0){\vertedge}\rput(3,11){\vertedge}
\rput(5,0){\vertedge}\rput(5,11){\vertedge}
\rput(4,1){\vertedge}\rput(4,10){\vertedge}
\rput(5,2){\vertedge}\rput(5,9){\vertedge}
\psline(0,0)(5,0)\psline(0,12)(5,12)
\psline(1,1)(5,1)\psline(1,11)(5,11)
\psline(2,2)(5,2)\psline(2,10)(5,10)
\psline(3,3)(5,3)\psline(3,9)(5,9)
\psline(4,4)(5,4)\psline(4,8)(5,8)
\multido{\i=1+1}{6}{\psarc{-*}(5,6){\i}{-90}{0}\psarc(5,6){\i}{0}{90}}
\psline(7,6)(8,6)\psline(9,6)(10,6)
\multido{\i=2+2}{5}{\rput(0,\i){\east}}
\multido{\i=3+2}{4}{\rput(1,\i){\east}}
\multido{\i=4+2}{3}{\rput(2,\i){\east}}
\multido{\i=5+2}{2}{\rput(3,\i){\east}}
\psline(4,6)(5,6)
}
\end{pspicture}

  \end{tabular}
  \caption{Fixed and non-fixed edges in odd size}
  \label{fig:fixed_odd}
\end{figure}

Figure~\ref{fig:fixed_odd} shows the fixed edges, and the fundamental
domain of non-fixed edges, for size $13$ ($n=6$), and
Figure~\ref{fig:honey_13} shows the same graph of non-fixed edges as a
region of the honeycomb lattice. In the latter figure, dotted edges
are those that are ``cut'' by Ciucu's Factorization Theorem, and bold
edges are those that are given a weight $1/2$ by the same.

\begin{figure}[ht]
\psset{unit=4mm,dotsep=1.5pt}
\begin{pspicture}(16,16)
\rput(8,8){
\psarc[linestyle=dotted]{o-*}(0,0){1}{0}{120}
\psarc[linestyle=dotted]{*-o}(0,0){2}{0}{120}
\psarc[linestyle=dotted]{o-*}(0,0){4}{0}{120}
\psarc[linestyle=dotted]{*-o}(0,0){5}{0}{120}
\psarc[linestyle=dotted]{o-*}(0,0){7}{0}{120}
\psarc[linestyle=dotted]{*-o}(0,0){8}{0}{120}
\psarc{*-o}(0,0){1}{-120}{0}
\psarc{o-*}(0,0){2}{-120}{0}
\psarc{*-o}(0,0){4}{-120}{0}
\psarc{o-*}(0,0){5}{-120}{0}
\psarc{*-o}(0,0){7}{-120}{0}
\psarc{o-*}(0,0){8}{-120}{0}
\psline[linewidth=.15]{o-*}(-1,0)(-2,0)
\psline[linewidth=.15]{o-*}(-4,0)(-5,0)
\psline[linewidth=.15]{o-*}(-7,0)(-8,0)
\psline[linewidth=.15]{o-*}(4,0)(2,0)
\psline[linewidth=.15]{o-*}(7,0)(5,0)
\rput(-8,0){\psline[linestyle=dotted]{*-o}(0,0)(1;240)}
\rput(-5,0){\psline[linestyle=dotted]{*-o}(0,0)(1;240)}
\rput(-2,0){\psline[linestyle=dotted]{*-o}(0,0)(1;240)}
\rput(-7,0){\psline[linestyle=dotted]{o-*}(0,0)(1;300)}
\rput(-4,0){\psline[linestyle=dotted]{o-*}(0,0)(1;300)}
\rput(-1,0){\psline[linestyle=dotted]{o-*}(0,0)(1;300)}

\multido{\i=1+3}{3}{\rput(\i;120){\nwest}}
\rput(-1,0){\multido{\i=2+3}{2}{\rput(\i;120){\nwest}}}
\rput(-2,0){\multido{\i=0+3}{3}{\rput(\i;120){\nwest}}}
\rput(-3,0){\multido{\i=-2+3}{3}{\rput(\i;120){\nwest}}}
\rput(-4,0){\multido{\i=2+3}{2}{\rput(\i;120){\nwest}}\rput(-4;120){\nwest}}
\rput(-5,0){\multido{\i=-3+3}{3}{\rput(\i;120){\nwest}}}
\rput(-6,0){\multido{\i=-5+3}{4}{\rput(\i;120){\nwest}}}
\rput(-7,0){\rput(2;120){\nwest}\multido{\i=-7+3}{2}{\rput(\i;120){\nwest}}}
\rput(-8,0){\multido{\i=-6+3}{3}{\rput(\i;120){\nwest}}}
\rput(-9,0){\multido{\i=-8+3}{3}{\rput(\i;120){\nwest}}}
\rput(-10,0){\multido{\i=-7+3}{2}{\rput(\i;120){\nwest}}}
\rput(-11,0){\rput(-6;120){\nwest}}

\rput(1;120){\multido{\i=-5+3}{2}{\rput(\i,0){\west}}}
\rput(2;120){\multido{\i=-6+3}{3}{\rput(\i,0){\west}}}
\rput(3;120){\multido{\i=-4+3}{2}{\rput(\i,0){\west}}}
\rput(4;120){\multido{\i=-5+3}{2}{\rput(\i,0){\west}}}
\rput(5;120){\multido{\i=-3+3}{2}{\rput(\i,0){\west}}}
\rput(6;120){\multido{\i=-4+3}{2}{\rput(\i,0){\west}}}
\rput(7;120){\rput(-2,0){\west}}
\rput(8;120){\west}
\rput(1;240){\multido{\i=-5+3}{2}{\rput(\i,0){\west}}}
\rput(2;240){\multido{\i=-6+3}{3}{\rput(\i,0){\west}}}
\rput(3;240){\multido{\i=-4+3}{2}{\rput(\i,0){\west}}}
\rput(4;240){\multido{\i=-5+3}{2}{\rput(\i,0){\west}}}
\rput(5;240){\multido{\i=-3+3}{2}{\rput(\i,0){\west}}}
\rput(6;240){\multido{\i=-4+3}{2}{\rput(\i,0){\west}}}
\rput(7;240){\rput(-2,0){\west}}
\rput(8;240){\west}

\multido{\i=2+3}{3}{\rput(\i;240){\neast}}
\rput(-1,0){\multido{\i=3+3}{2}{\rput(\i;240){\neast}}}
\rput(-2,0){\multido{\i=4+3}{2}{\rput(\i;240){\neast}}}
\rput(-3,0){\multido{\i=-1+3}{3}{\rput(\i;240){\neast}}}
\rput(-4,0){\multido{\i=-3+3}{4}{\rput(\i;240){\neast}}}
\rput(-5,0){\multido{\i=-2+6}{2}{\rput(\i;240){\neast}}}
\rput(-6,0){\multido{\i=-4+3}{4}{\rput(\i;240){\neast}}}
\rput(-7,0){\multido{\i=-6+3}{4}{\rput(\i;240){\neast}}}
\rput(-8,0){\multido{\i=-5+3}{2}{\rput(\i;240){\neast}}}
\rput(-9,0){\multido{\i=-7+3}{3}{\rput(\i;240){\neast}}}
\rput(-10,0){\multido{\i=-6+3}{2}{\rput(\i;240){\neast}}}
\rput(-11,0){\rput(-5;240){\neast}}

\rput(-1,0){\neast}
}
\end{pspicture}
  
  \caption{Honeycomb version of the graph in Figure~\ref{fig:fixed_odd}}
  \label{fig:honey_13}
\end{figure}
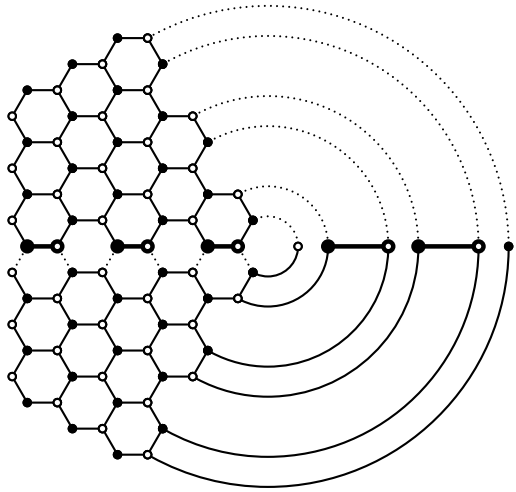

When we restrict our attention to the nonfixed edges, the
corresponding HTFPLs are in bijection with the perfect matchings of a
region $G_n$ of the honeycomb lattice as shown on Figure~\ref{fig:gn}
(where the sides of the region along the bold line must be glued
together). 

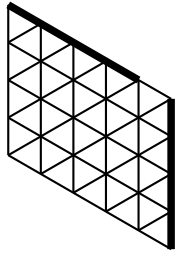
\begin{figure}[ht]
  \begin{center}
\psset{unit=5mm}
    \begin{pspicture}(0,0)(5,8)
\SpecialCoor
\rput(5,5){\multido{\i=0+1}{6}{\rput(\i;150){\psline(0,0)(0,-4)}}
   \multido{\i=0+1}{5}{\rput(\i;270){\psline(0,0)(5;150)}}
   \rput(4;150){\psline(0,0)(1;210)}
   \rput(3;150){\psline(0,0)(2;210)}
   \rput(2;150){\psline(0,0)(3;210)}
   \rput(1;150){\psline(0,0)(4;210)}
   \psline(0,0)(4;210)
   \rput(1;270){\psline(0,0)(3;210)}
   \rput(2;270){\psline(0,0)(2;210)}
   \rput(3;270){\psline(0,0)(1;210)}
\psline[linewidth=.2](1;150)(5;150)
\psline[linewidth=.2](0,0)(0,-4)
}
    \end{pspicture}
  \end{center}
\caption{The region $G_4$.}
\label{fig:gn}
\end{figure}

The region $G_n$ can be deformed to have a reflexive
symmetry as shown on
Figure \ref{fig:RR} size $4k+1$ and $4k+3$.

\begin{figure}[htbp]

\begin{tabular}{cc}

  \psset{unit=5mm}
  \begin{pspicture}(6,10)
    \SpecialCoor
    \rput(5,5){
      \rput(0,-1){\pspolygon[fillstyle=hlines](0,0)(1;30)(1;90)(1;150)}
      \rput(6;210){\multido{\i=0+1}{3}{\rput(0,\i){\pspolygon[fillstyle=hlines](0,0)(1;330)(1;30)(1;90)}}}
      \rput(6;210){\multido{\i=1+1}{3}{\rput(\i;330){\pspolygon[fillstyle=hlines](0,0)(1;330)(1;30)(1;90)}}}
      \rput(1;210){\pspolygon[fillstyle=solid,fillcolor=lightgray](0,0)(1;30)(1;90)(1;150)(0,0)}
      \rput(2;210){\rput(1;150){\pspolygon[fillstyle=solid,fillcolor=lightgray](0,0)(1;30)(1;90)(1;150)(0,0)}}
      \rput(3;210){\rput(2;150){\pspolygon[fillstyle=solid,fillcolor=lightgray](0,0)(1;30)(1;90)(1;150)(0,0)}}
      \rput(0,-2){\pspolygon[fillstyle=solid,fillcolor=lightgray](0,0)(1;90)(1;150)(1;210)(0,0)}
      \rput(0,-3){\rput(1;210){\pspolygon[fillstyle=solid,fillcolor=lightgray](0,0)(1;90)(1;150)(1;210)(0,0)}}
      \rput(6;210){\psline(0,0)(0,6)\psline(0,0)(4;330)}
      \rput(5;210){\psline(0,-1)(0,6)\psline(1;150)(3;330)}
      \rput(4;210){\psline(0,-2)(0,6)\psline(2;150)(3;330)}
      \rput(3;210){\psline(0,-3)(0,6)\psline(3;150)(2;330)}
      \rput(2;210){\psline(0,-4)(0,6)\psline(4;150)(2;330)}
      \rput(1;210){\psline(0,-3)(0,4)\psline(5;150)(1;330)}
      \psline(0,-2)(0,2)\psline(6;150)(1;330)
      \rput(0,1){\psline(5;150)(1;330)}
      \rput(0,2){\psline(4;150)(0,0)}
      \rput(0,3){\psline(3;150)(1;150)}
      \rput(6;150){\psline(0,0)(4;30)}
      \rput(5;150){\psline(1;210)(3;30)}
      \rput(4;150){\psline(2;210)(3;30)}
      \rput(3;150){\psline(3;210)(2;30)}
      \rput(2;150){\psline(4;210)(2;30)}
      \rput(1;150){\psline(5;210)(1;30)}
      \psline(6;210)(1;30)
      \rput(0,-1){\psline(5;210)(1;30)}
      \rput(0,-2){\psline(4;210)(0,0)}
      \rput(0,-3){\psline(3;210)(1;210)}
      }
    \rput(5,5){
      \psline[linewidth=.2](1;210)(0,0)(0,1)
      \rput(1;150){\psline[linewidth=.2](0,-1)(1;210)(2;210)
        \psline[linewidth=.2](1;30)(0,1)(0,2)}
      \rput(2;150){\rput(0,2){\psline[linewidth=.2](0,0)(0,1)
                              \psline[linewidth=.2](0,0)(1;330)}
                   \rput(2;210){\psline[linewidth=.2](0,0)(1;330)
                                \psline[linewidth=.2](0,0)(1;210)}}
      \rput(0,3){\psline[linewidth=.2](2;150)(3;150)}
      \rput(0,-3){\psline[linewidth=.2](6;150)(5;150)}
      \rput(0,4){$R_k$}\rput(0,-4){$R'_{k-1}$}
}
  \end{pspicture}
& 
 \psset{unit=5mm}
 \begin{pspicture}(10,10)
    \SpecialCoor
    \rput(9,5){
      \rput(6;150){\multido{\i=0+1}{4}{\rput(\i;30){\pspolygon[fillstyle=hlines](0,0)(1;90)(1;150)(1;210)}}}
      \rput(6;150){\multido{\i=1+1}{3}{\rput(\i;270){\pspolygon[fillstyle=hlines](0,0)(1;90)(1;150)(1;210)}}}
      \pspolygon[fillstyle=hlines](0,0)(1;210)(1;270)(1;330)
      \rput(1;210){\pspolygon[fillstyle=solid,fillcolor=lightgray](0,0)(1;30)(1;90)(1;150)(0,0)}
      \rput(2;210){\rput(1;150){\pspolygon[fillstyle=solid,fillcolor=lightgray](0,0)(1;30)(1;90)(1;150)(0,0)}}
      \rput(3;210){\rput(2;150){\pspolygon[fillstyle=solid,fillcolor=lightgray](0,0)(1;30)(1;90)(1;150)(0,0)}}
      \rput(0,-2){\pspolygon[fillstyle=solid,fillcolor=lightgray](0,0)(1;90)(1;150)(1;210)(0,0)}
      \rput(0,-3){\rput(1;210){\pspolygon[fillstyle=solid,fillcolor=lightgray](0,0)(1;90)(1;150)(1;210)(0,0)}}
      \rput(0,-4){\rput(2;210){\pspolygon[fillstyle=solid,fillcolor=lightgray](0,0)(1;90)(1;150)(1;210)(0,0)}}
      \rput(6;210){\psline(0,0)(0,6)\psline(0,0)(4;330)}
      \rput(5;210){\psline(0,-1)(0,6)\psline(1;150)(3;330)}
      \rput(4;210){\psline(0,-2)(0,6)\psline(2;150)(3;330)}
      \rput(3;210){\psline(0,-3)(0,6)\psline(3;150)(2;330)}
      \rput(2;210){\psline(0,-4)(0,6)\psline(4;150)(2;330)}
      \rput(1;210){\psline(0,-3)(0,4)\psline(5;150)(1;330)}
      \psline(0,-2)(0,2)\psline(6;150)(1;330)
      \rput(0,1){\psline(5;150)(1;330)}
      \rput(0,2){\psline(4;150)(0,0)}
      \rput(0,3){\psline(3;150)(1;150)}
      \rput(6;150){\psline(0,0)(4;30)}
      \rput(5;150){\psline(1;210)(3;30)}
      \rput(4;150){\psline(2;210)(3;30)}
      \rput(3;150){\psline(3;210)(2;30)}
      \rput(2;150){\psline(4;210)(2;30)}
      \rput(1;150){\psline(5;210)(1;30)}
      \psline(6;210)(1;30)
      \rput(0,-1){\psline(5;210)(1;30)}
      \rput(0,-2){\psline(4;210)(0,0)}
      \rput(0,-3){\psline(3;210)(1;210)}
      \rput(7;150){
        \psline(1;330)(0,0)(4;30)
        \multido{\i=1+1}{4}{\rput(\i;30){\psline(0,-1)(0,0)(1;330)}}}
      \rput(7;210){
        \psline(0,0)(4;330)
        \multido{\i=1+1}{4}{\rput(\i;330){\psline(0,1)(0,0)(1;30)}}
        \psline(1;30)(0,0)(0,7)
        \multido{\i=1+1}{6}{\rput(0,\i){\psline(1;30)(0,0)(1;330)}}}
      }
    \rput(9,5){
      \psline[linewidth=.2](1;210)(0,0)(0,1)
      \rput(1;150){\psline[linewidth=.2](0,-1)(1;210)(2;210)
        \psline[linewidth=.2](1;30)(0,1)(0,2)}
      \rput(2;150){\rput(0,2){\psline[linewidth=.2](0,0)(0,1)
                              \psline[linewidth=.2](0,0)(1;330)}
                   \rput(2;210){\psline[linewidth=.2](0,0)(1;330)
                                \psline[linewidth=.2](0,0)(1;210)}}
      \rput(0,3){\psline[linewidth=.2](2;150)(3;150)}
      \rput(0,-3){\psline[linewidth=.2](6;150)(5;150)}
      \rput(3;150){\psline[linewidth=.2](0,3)(0,4)
                   \psline[linewidth=.2](3;210)(4;210)}
      \rput(0,4){$R_k$}
      \rput(0,-4){$R'_{k}$}
}
  \end{pspicture}

\end{tabular}
\caption{Decomposition by symmetry of $G_6$ and $G_7$.}
\label{fig:RR}
\end{figure}
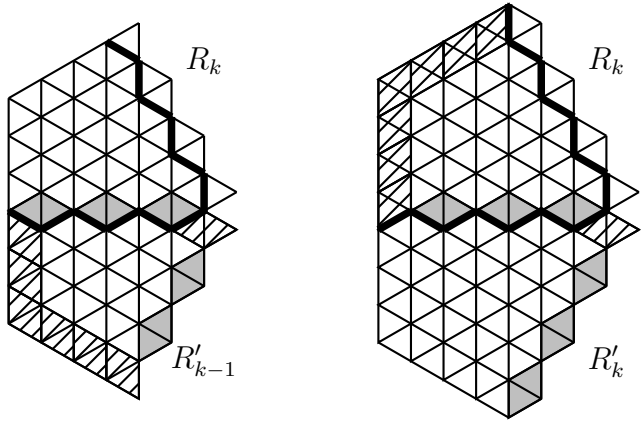

Figure~\ref{fig:RR} shows (in the triangular lattice) the result of
applying Ciucu's Factorization Theorem~\cite{ciucu1}: for each size,
we are to count the lozenge tilings of two regions of the trangular
lattice. In the figure, grayed lozenges have a weight of $1/2$
attached, and dashed lozenges are ``fixed'' in the sense that they
must appear in all lozenge tilings of the corresponding region.

Ciucu's theorem thus implies that the number $H_n$ of such HTFPLs of size $n$ is given in odd size by:
  \begin{itemize}
  \item $H_{4k+1} = 2^{2k} R_k(1/2,1) R'_{k-1}(1/2,1)$,
  \item $H_{4k+3} = 2^{2k+1} R_k(1/2,1) R'_k(1/2,1)$
  \end{itemize}

and in even size by:

  \begin{itemize}
  \item $H_{4k} = 2^{2k} R_k(1/2,1/2) R_{k-1}(1,1)$,
  \item $H_{4k+2} = 2^{2k+2} R_k(1/2,1/2) R_{k-1}(1,1)$. 
  \end{itemize}

To enumerate weighted tilings of regions $R_k$ and $R'_k$, we may use Lindström-Gessel-Viennot's \cite{GV} determinants to get:

  \begin{eqnarray*}
    R_k(x,y) &=& \det\left( m_{i,j,0}\right)_{1\leq i,j\leq k}\\
    R'_k(x,y) &=& \det\left( m_{i,j,1} \right)_{1\leq i,j\leq k}\\
m_{i,j,\ell} & = & (1\!+\!xy)\!\binom{i\!+\!j\!+\ell\!-\!2}{2i\!-\!j\!-\!1} + x\binom{i\!+\!j\!+\!\ell\!-\!2}{2i\!-\!j\!-\!2} + y\binom{i\!+\!j\!+\!\ell\!-\!2}{2i\!-\!j}
  \end{eqnarray*}

\subsection{Enumeration of certain tilings of hexagons}

When evaluating $R_k(x,y)$ and $R'_k(x,y)$ for $x$ and $y$ in  $\{1/2,1\}$, 
we are surprised to recover well-known sequences.
More generally, we define:
$$    \R_\ell(n;x,y) = \det\left( m_{i,j,\ell}\right)_{1\leq i,j\leq n}.$$
The aim of this subsection is to identify some specializations of the functions $\R_\ell(n;x,y)$ in terms of cardinality of some classes of alternating sign matrices. 

The function $\mathcal{R}_{\ell}(n;x,y)$ counts the weighted lozenge
tilings of the region shown in Figure~\ref{fig:Rl}, where grayed
lozenges carry a multiplicative weight of $x$ or $y$ as indicated.

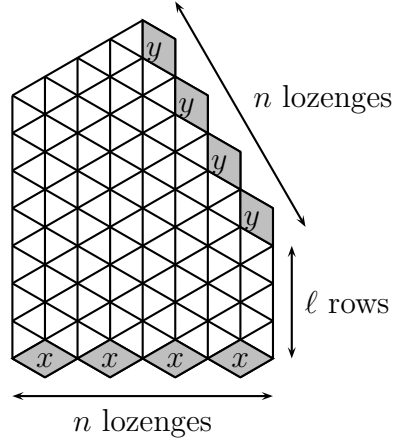
\begin{figure}[ht]
\psset{unit=5mm}
\begin{pspicture}(8,10)
  \rput(7,1){
    \psline(0,0)(0,4)
    \rput(1;210){\psline(0,0)(0,6)}
    \rput(2;210){\psline(0,1)(0,8)}
    \rput(3;210){\psline(0,1)(0,10)}
    \rput(4;210){\psline(0,2)(0,11)}
    \rput(5;210){\psline(0,2)(0,11)}
    \rput(6;210){\psline(0,3)(0,11)}
    \rput(7;210){\psline(0,3)(0,11)}
    \rput(8;210){\psline(0,4)(0,11)}

    \multido{\i=0+1}{4}{\rput(0,\i){\psline(0,0)(8;150)}}
    \rput(0,4){\psline(0,0)(7;150)}
    \rput(0,5){\psline(1;150)(6;150)}
    \rput(0,6){\psline(2;150)(5;150)}
    \rput(0,7){\psline(3;150)(4;150)}
    \rput(0,-1){\psline(1;150)(8;150)}
    \rput(0,-2){\psline(3;150)(8;150)}
    \rput(0,-3){\psline(5;150)(8;150)}
    \rput(0,-4){\psline(7;150)(8;150)}

    \rput(0,0){\psline(0,0)(1;210)}
    \rput(0,1){\psline(0,0)(3;210)}
    \rput(0,2){\psline(0,0)(5;210)}
    \rput(0,3){\psline(0,0)(7;210)}
    \rput(0,4){\psline(0,0)(8;210)}
    \rput(0,5){\psline(1;210)(8;210)}
    \rput(0,6){\psline(1;210)(8;210)}
    \rput(0,7){\psline(2;210)(8;210)}
    \rput(0,8){\psline(2;210)(8;210)}
    \rput(0,9){\psline(3;210)(8;210)}
    \rput(0,10){\psline(3;210)(8;210)}
    \rput(0,11){\psline(4;210)(8;210)}

    \psline{<->}(.5,0)(.5,3)
    \uput[0](.5,1.5){$\ell$ rows}
    \psline{<->}(-6.93,-1)(0,-1)
    \uput[270](-3.46,-1){$n$ lozenges}
    \rput(0,3){\rput(1;30){\psline{<->}(0,0)(6.93;120)\uput[30](3.46;120){$n$ lozenges}}}

    \rput(1;150){\pspolygon[fillstyle=solid,fillcolor=lightgray](0,0)(1;210)(1;270)(1;330)\uput[270](0,0){$x$}}
    \rput(2;150){\rput(1;210){\pspolygon[fillstyle=solid,fillcolor=lightgray](0,0)(1;210)(1;270)(1;330)\uput[270](0,0){$x$}}}
    \rput(3;150){\rput(2;210){\pspolygon[fillstyle=solid,fillcolor=lightgray](0,0)(1;210)(1;270)(1;330)\uput[270](0,0){$x$}}}
    \rput(4;150){\rput(3;210){\pspolygon[fillstyle=solid,fillcolor=lightgray](0,0)(1;210)(1;270)(1;330)\uput[270](0,0){$x$}}}

    \rput(0,4){
      \pspolygon[fillstyle=solid,fillcolor=lightgray](0,0)(1;150)(1;210)(1;270)\uput[210](0,0){$y$}
      \rput(1;150){\rput(0,1){\pspolygon[fillstyle=solid,fillcolor=lightgray](0,0)(1;150)(1;210)(1;270)\uput[210](0,0){$y$}}}
      \rput(2;150){\rput(0,2){\pspolygon[fillstyle=solid,fillcolor=lightgray](0,0)(1;150)(1;210)(1;270)\uput[210](0,0){$y$}}}
      \rput(3;150){\rput(0,3){\pspolygon[fillstyle=solid,fillcolor=lightgray](0,0)(1;150)(1;210)(1;270)\uput[210](0,0){$y$}}}
    }

}
\end{pspicture}
\caption{Interpretation of $\mathcal{R}_{\ell}(n;x,y)$}
\label{fig:Rl}
\end{figure}

\begin{proposition}
We have the following special values for the functions $\R$:
\begin{align}
\label{eq1}\R_0(n;1/2,1)&=A_{HT}(2n+1)\\
\label{eq2}\R_1(n;1/2,1)&=\frac{1}{2}A_{HT}(2n+2)\\
\label{eq3}\R_1(n;1,1)&=A_{V}(2n+3)\\
\label{eq4}\R_1(n;1,1/2)&=A(n)^2\\
\label{eq5}\R_2(n;1/2,1)&=A(n)A(n+1)
\end{align}
where $A_{HT}(N)$, $A_{V}(N)$ and $A(N)$ stand respectively for the number
of half-turn symmetric, vertically symmetric and unrestricted alternating sign 
matrices.
\end{proposition}

\Proof
It appears that the three specializations we need to interpret, namely $\R_\ell(n;1,1/2)$, $\R_\ell(n;1/2,1)$ and $\R_\ell(n;1,1)$ are computed in \cite{ciucu5,kratt}. 
We denote by $(a)_i=a(a+1)\cdots(a+i-1)$ the shifted factorial.

\medskip
\noindent {\em Proof of  \pref{eq1} and \pref{eq2}.}
We may use \cite{kratt} to write:
$$\R_\ell(n;1/2,1) = \prod_{i=0}^{n-1}(2\ell+3i) i! \frac {(\ell+i-1)! (2\ell+2i)_i (\ell+2i)_i} {(\ell+2i)! (2i)!} .$$

It is then  a simple computation to check that 
\begin{itemize}
\item for $\ell=0$: 
$$(3i) i! \frac {(i-1)! (2i)_i^2} {(2i)!^2} = \frac 4 3 \frac{{{3n}\choose{n}}^2}{{{2n}\choose{n}}^2} = \frac{A_{HT}(2n+1)}{A_{HT}(2n-1)},$$
\item for $\ell=1$: 
$$(2+3i) i! \frac {(i)! (2i+2)_i (2i+1)_i} {(2i+1)! (2i)!}  = \frac 4 3 \frac{{{3n+3}\choose{n+1}} {{3n}\choose{n}}}  {{{2n+2}\choose{n+1}} {{2n}\choose{n}}} = \frac{A_{HT}(2n)}{A_{HT}(2n)}.$$
\end{itemize}

Now to conclude, we observe that $\R_0(1;1/2,1)=3=A_{HT}(3)$ and $\R_0(1;1/2,1)=5=1/2.A_{HT}(4)$. This implies equations \pref{eq1} and \pref{eq2}.

\medskip
\noindent {\em Proof of \pref{eq3}.} 
We know from \cite{ciucu5} that $\R_1(n;1,1)$ is equal to the number of cyclically symmetric transpose-complementary plane partitions (CSTCPP) in a hexagonal region with a triangular hole of size $2$. We thus get:
$$\R_1(n;1,1)=P_{CSTC}(2n,2)=\frac{1}{2^n}\prod_{j=0}^{n-1} \frac{P_{CS}(2j+1,2)}{P_{CS}(2j,2)}$$
By using 
\begin{align*}
P_{CS}(2j+1,2)&= \frac{(-1/2)! (2j+3)_{j+1}}{(j+1/2)!}\\
&\times \prod_{i=0}^j \frac{i!^2 (2i+1)_i^2 (i+1/2)! (2i+1/2)_{i+1} (2i+1+1/2)_i}{(2i)!^2(j+i+1+1/2)!}\\
P_{CS}(2j,2)&=\frac{(-1/2)! j! (2j+1/2)_{j+1}}{(2j)! (2j+1/2)!} \\
&\times \prod_{i=0}^{j-1} \frac{i!^2 (2i+3)_{i+1}^2 (i+1/2)! (2i+1+1/2)_{i} (2i+1/2)_{i+1}}{(2i)!^2(j+i+1/2)!}
\end{align*}
we get:
$$\R_1(n;1,1) = \frac{1}{2^n} \prod_{j=0}^{n-1} \frac{j! (2j+1+1/2)_{j} (2j)!^2 (2j+1)_{j}}{(3j)!^2 (j+1+1/2)_{j+1}}.$$
Thus, because of \cite{robbins,kup}
$$A_V(2n+1)= (-3)^{n^2}\prod_{i, j\le2n+1,j\equiv1[2]} \frac {3( j-i)+1} {j-i+2n+1} = \prod_{j=1}^{n} \frac{ {{6j-2}\choose{2j}} }{ {{4j-1}\choose{2j}} },$$
we have to check that:
$$\frac{j! (2j+1+1/2)_{j} (2j)!^2 (2j+1)_{j}}{(3j)!^2 (j+1+1/2)_{j+1}} = \frac{ {{6j+4}\choose{2j+2}} }{ {{4j+3}\choose{2j+2}} }$$
which comes from a simple computation.

\medskip
\noindent {\em Proof of \pref{eq4} and \pref{eq5}.}
For equation \pref{eq4}, we know from \cite{ciucu5} that $\R_1(n;1,1/2)$ is equal to the number of cyclically symmetric 
self-complementary plane partitions (CSSCPP) in a hexagonal region, which is known \cite{kupPP} to be given by: 
$$P_{CSSC}(2n)=\Big(\prod_{i=0}^{n-1} \frac{(3i+1)!}{(n+i)!} \Big)^2
=A(n)^2.$$

\medskip
For equation \pref{eq5}, it has been shown in \cite{duchon} that $\R_2(n;1/2,1)$ is the number of quasi-cyclically symmetric self-complementary plane partitions (qCSSCPP) in a hexagonal region, which is proved to be given by: 
$$P_{qCSSC}(2n+1)=A(n)A(n+1).$$

\qed

\begin{remark}
It appears that the specialization of the functions $\R_\ell(n;x,y)$ to $x=y=1/2$ may also have interesting values. In particular, it seems that $\R_0(n;1/2,1/2)$ corresponds to the development of the generating series for $A_{UU}^{(2)}(4n)$ ({\it cf.} \cite{kup}) and that:
$$\R_2(n;1/2,1/2)=A_V(2n+3){{2n+1}\choose{n+1}}$$
which needs an explanation.
\end{remark}

%%%%%%%%%%%%%%%%%%%%%%%%%%%%%%%%%%%%%%%%%%%%%%%%%%%%%%%%%%%%%%%%%%%%%%%%%%%%%%%
%%%%%%%%%%%%%%%%%%%%%%%%%%%%%%%%%%%%%%%%%%%%%%%%%%%%%%%%%%%%%%%%%%%%%%%%%%%%%%%
%%%%%%%%%%%%%%%%%%%%%%%%%%%%%%%%%%%%%%%%%%%%%%%%%%%%%%%%%%%%%%%%%%%%%%%%%%%%%%%

\end{document}